\documentclass[article,12pt]{amsart}

\usepackage{amsthm,amscd,amsfonts}
\usepackage{amssymb, upref}
\usepackage{latexsym}
\usepackage{amsmath,amsthm,amsfonts,amssymb}
\usepackage[pdftex]{graphicx}
\usepackage{tikz}
\usepackage{latexsym}
\usepackage{amscd}
\usepackage[all]{xy}
\usepackage{xy}
\usepackage{bm}

\newtheorem{theorem}{Theorem}[section]
\newtheorem{lemma}[theorem]{Lemma}
\newtheorem{corollary}[theorem]{Corollary}
\newtheorem{conjecture}[theorem]{Conjecture}
\newtheorem{question}[theorem]{Question}
\newtheorem{proposition}[theorem]{Proposition}
\theoremstyle{definition}
\newtheorem{definition}[theorem]{Definition}
\newtheorem{remark}[theorem]{Remark}
\newtheorem{example}[theorem]{Example}
\def\bal{\begin{array}{ll}}
\def\eal{\end{array}}

\def\scs{\scriptstyle}
\def\scscs{\scriptscriptstyle}

\numberwithin{equation}{section}

\title[A Plethysm formula]{A Plethysm formula on the characteristic map
of induced linear characters from $U_n(\mathbb F_q)$ to
$GL_n(\mathbb F_q)$}

\author[Z. Chen]
{Zhi Chen}

\address[Zhi Chen]
{Department of Mathematics and Statistics\\ York University\\
   To\-ron\-to, Ontario M3J 1P3\\ CANADA}
\email[Zhi Chen]{czhi@mathstat.yorku.ca}
\date{\today}

\begin{document}

\begin{abstract} This paper gives a plethysm formula on the
characteristic map of the induced linear characters from the
unipotent upper-triangular matrices $U_n(\mathbb F_q)$ to
$GL_n(\mathbb F_q)$, the general linear group over finite field
$\mathbb F_q$. The result turns out to be a multiple of a twisted
version of the Hall-Littlewood symmetric functions
$\tilde{P}_n(Y,q)$. A recurrence relation is also given which makes
it easy to carry out the computation.
\end{abstract}

\maketitle
\section{Introduction}
Let $\mathbb F_q$ be a fixed finite field and $GL_n(\mathbb F_q)$
the finite general linear group over $\mathbb F_q$. The
representation theory of $GL_n(\mathbb F_q)$ over $\mathbb{C}$ has
been thoroughly studied by J.A.Green~\cite{Green}. He also
constructed the characteristic map which builds a connection between
the character spaces of $GL_n(\mathbb F_q)$ for $n\geq 0$ and the
Cartesian product over infinitely indexed sets of rings of symmetric
functions. In character theory, the study of induced linear
characters from subgroups is very useful to understand the character
ring of the larger group. In this paper, we consider certain induced
linear characters from the group of unipotent upper-triangular
matrices $U_n(\mathbb F_q)$ to $GL_n(\mathbb F_q)$. The
representations of these induced linear characters are known as
Gelfand-Graev modules, which play an important role in the
representation theory of finite groups of Lie type~(\cite{Gelgra},
\cite{Stein}). The formula for the characteristic map of the induced
linear characters is given by Thiem~\cite{Nat}. We then apply
plethysms on the image of the characteristic map. There are two
advantages in doing so: to get a simpler formula and to express the
result as a multiple of a twisted version of the Hall-Littlewood
symmetric functions $\tilde{P}_n(y,q)$. We hope this method could
contribute to the study on the irreducible decomposition of the
induced characters from $U_n(\mathbb F_q)$ to $GL_n(\mathbb F_q)$.

In section $2$ we give some background knowledge on symmetric
functions and representation theory of $GL_n(\mathbb F_q)$ and
$U_n(\mathbb F_q)$. Since the character theory of $U_n(\mathbb F_q)$
is known as a wild problem, supercharacter theory is built up as an
approximation of the ordinary character theory. The linear
characters of $U_n(\mathbb F_q)$ that we are considering are part of
the category of supercharcters of $U_n(\mathbb F_q)$.  We introduce
further questions about the induction of all supercharacters in
Section $4$. In Section $3$ we give our main result about the
plethysm formula. A recurrence relation is obtained naturally so
that we can carry out the computation of plethysms on the
characteristic map of the induced linear characters more easily. We
also give a relation between the characteristic map of the induced
characters from $U_n(\mathbb F_q)$ to $GL_n(\mathbb F_q)$ and the
plethysms on those characteristics. This is depicted in the
following diagram \xyoption{all}
\begin{displaymath}
\xymatrix{
  \otimes_{\varphi\in\Theta} \Lambda_{\mathbb C}(Y^{\varphi}) \ar[d]_{\rho} \ar[r]^{T}
                & \otimes_{f\in\Phi} \Lambda_{\mathbb C}(X_f) \ar[d]^{\Pi |_{\Lambda_{\mathbb C}(X_{f=x-1})}}  \\
 \Lambda_{\mathbb C}(Y)           \ar[r]_{t\circ\omega}
                & \Lambda_{\mathbb C}(X_{x-1})}
\end{displaymath}
where the notation is explained in Theorem~\ref{thm:comm}. From the
above commutative diagram we show that our simplified plethysm
formula does not lose any information on the characteristic map of
the induced characters from $U_n(\mathbb F_q)$ to $GL_n(\mathbb
F_q)$.

\noindent{\textbf{Acknowledgement}}\\
I am grateful to my advisor, Professor Nantel Bergeron, for his
guidance and discussions on this research problem.

\section{Background}
\subsection{Symmetric functions}

The notation in this paper follows closely the book of Macdonald
\cite{Mac}.
\begin{definition}
A {\sl partition} $\lambda$ of $n\in  \textbf{N}$, is a sequence
$\lambda=(\lambda_1, \lambda_2, \ldots, \lambda_l)$ of positive
integers in weakly decreasing order:
$\lambda_1\geq\lambda_2\geq\cdots\geq\lambda_l$, such that
$\lambda_1+\lambda_2+\cdots+\lambda_l=n$. We denote this by
$\lambda\vdash n$. Here, each $\lambda_i\ (1\leq i\leq l)$ is called
a {\sl part} of $\lambda$. We say the {\sl length} of the partition
$\lambda$ is $l$, which is the number of parts of $\lambda$. We use
$|\lambda|$ to denote the sum of all parts
$\lambda_1+\lambda_2+\cdots+\lambda_l$, and we call $|\lambda|$ the
{\sl size} of the partition. Sometimes we also use the notation:
$$\lambda = (1^{m_1},2^{m_2},\ldots,n^{m_n},\ldots)$$
where each $m_i$ means there are $m_i$ parts in $\lambda$ equal to
$i$.
\end{definition}

Let $\Lambda_{\mathbb C}(Y)$ denote the ring of symmetric functions
with complex coefficients in the variables $Y=\{y_1,y_2,\ldots\}$.
We denote the complete symmetric functions, elementary symmetric
functions, monomial symmetric functions, power-sum symmetric
functions, and Schur symmetric functions by $h_{\lambda}(Y)$,
$e_{\lambda}(Y)$, $m_{\lambda}(Y)$, $p_{\lambda}(Y)$, and
$s_{\lambda}(Y)$ respectively.

The generating function for $h_n(Y)$ is
$$H(Y;t) = \sum_{n\geq 0} h_n(Y) t^n = \prod_{j\geq 1} (1-y_j t)^{-1}.$$

Let $X=\{x_1,x_2,\ldots\}$ be another set of finite or infinite
variables. We have the following identity:
\begin{equation}\label{eq:omg}
\Omega[XY]:=\prod_{i,j}(1-x_iy_j)^{-1} =
\sum_{\lambda}m_{\lambda}(X)h_{\lambda}(Y)
\end{equation}
summed over all partitions $\lambda$.

There is a scalar product defined on $\Lambda_{\mathbb C}(Y)$, which
makes $(m_{\lambda})$ and $(h_{\lambda})$ dual to each other:
$$\langle h_{\lambda}, m_{\mu} \rangle = \delta_{\lambda\mu}$$
for all partitions $\lambda, \mu$, where $\delta_{\lambda\mu}$ is
the Kronecker delta.

We use $P_{\lambda}(Y;t)$ to denote the Hall-Littlewood symmetric
functions as defined in \cite{Mac}. If we define
\begin{align*}
q_r &= q_r(Y;t) = (1-t)P_{(r)}(Y; t)\ \ \text{for}\ r\geq 1 ,\\
q_0 &= q_0(Y;t) = 1 ,
\end{align*}
then the generating function for $q_r(Y;t)$ is
\begin{equation}\label{eq:genqr}
Q(u) = \sum_{r\geq 0} q_r(Y; t) u^r = \prod_i\frac{1-y_i t u}{1-y_i
u}\, \lower 6pt\hbox{.}
\end{equation}
For each partition $\lambda$ let $n(\lambda) = \sum_{i\geq 1}
(i-1)\lambda_i $. Define
\begin{equation*}
\tilde{P}_{\lambda}(Y;q) = q^{-n(\lambda)} P_{\lambda}(Y;q^{-1})
\end{equation*}
and we call $\tilde{P}_{\lambda}(Y;q)$ the twisted Hall-Littlewood
symmetric functions.

From \cite{Mac} it is well known that the {\sl plethysm} can be
defined by
\begin{equation}\label{eq:defplethysm}
h_a[ p_b ] = h_a(y_1^b,y_2^b,\ldots),
\end{equation}
which is the coefficient of $t^{ab}$ in $\prod_{j\geq 1} (1-y_j^b
t^b)^{-1}.$

\subsection{ Representation theory of $GL_n(\mathbb F_q)$}
The representation theory of the finite general linear group
$G_n=GL_n(\mathbb F_q)$ over $\mathbb{C}$ can be found in
J.A.Green~\cite{Green}, Macdonald~\cite{Mac} and Thiem~\cite{Nat}.
Here we give a short description on the characteristic map
constructed by J.A.Green.

Let $\bar{\mathbb F}_q$ denote the algebraic closure of the finite
field $\mathbb F_q$. The multiplicative group of $\bar{\mathbb F}_q$
is denoted by $\bar{\mathbb F}_q^{\times}$. Let $\bar{\mathbb
F}_q^{\ast}=\{\phi: \bar{\mathbb F}_q^{\times} \rightarrow \mathbb
C^{\times}\}$ be the group of complex-valued multiplicative
characters of $\bar{\mathbb F}_q^{\times}$. The Frobenius
automorphism of $\bar{\mathbb F}_q$ over $\mathbb F_q$ is given by
$$F: x \rightarrow x^q, \text{where}\ x\in \bar{\mathbb F}_q.$$

We then define
$$\Phi=\{F\text{-orbits\ of}\ \bar{\mathbb F}_q^{\times}\}\ \text{and}\
\Theta=\{F\text{-orbits\ of}\ \bar{\mathbb F}_q^{\ast}\}.$$ Since
every $F\text{-orbits\ of}\ \bar{\mathbb F}_q^{\times}$ is in
one-to-one correspondence with irreducible polynomial $f$ over
$\mathbb F_q$, we can also use $f$ to denote each $F\text{-orbit}$
in $\Phi$. A partition-valued function $\bm{\mu}$ on $\Phi$ is a
function which maps each $f\in\Phi$ to a partition $\bm{\mu}(f)$.
The size of $\bm{\mu}$ is
$$\|\bm{\mu}\|=\sum_{f\in\Phi} d(f)|\bm{\mu}(f)|,$$
where $d(f)$ is equal to the degree of $f\in\Phi$.

Let $\mathbb P$ denote the set of all partitions and
\begin{equation*}
\mathcal{P}^{\Phi}=\bigcup_{n\geq 0}\mathcal{P}_n^{\Phi},\
\text{where}\ \mathcal{P}_n^{\Phi}=\{\bm{\mu}: \Phi \rightarrow
\mathbb{P}\,;\ \|\mbox{\boldmath{$\mu$}}\|=n \}.
\end{equation*}
We use $K^{\bm{\mu}}$ to denote the conjugacy classes in $G_n$
parameterized by $\bm{\mu}\in\mathcal{P}_n^{\Phi}$~\cite{Mac}. The
characteristic function of the conjugacy class $K^{\bm{\mu}}$ is
denoted by $\pi_{\bm{\mu}}$.

Similarly, for each partition-valued function
$\mbox{\boldmath{$\lambda$}}: \Theta \rightarrow \mathbb{P}$, the
size of $\bm{\lambda}$ is
$$\|\bm{\lambda}\|=\sum_{\varphi\in\Theta} d(\varphi)|\bm{\lambda}(\varphi)|,$$
where $d(\varphi)$ is equal to the number of elements in $\varphi$.
Let
\begin{equation*}
\mathcal{P}^{\Theta}=\bigcup_{n\geq 0}\mathcal{P}_n^{\Theta},\
\text{where}\ \mathcal{P}_n^{\Theta}=\{\bm{\lambda}: \Theta
\rightarrow \mathbb{P}\,;\ \|\bm{\lambda}\|=n \}.
\end{equation*}
We use $G_n^{\bm{\lambda}}$ to denote the irreducible $G_n$-modules
indexed by $\bm{\lambda}\in\mathcal{P}_n^{\Theta}$ \cite{Mac}. The
character of the irreducible $G_n$-modules $G_n^{\bm{\lambda}}$ is
denoted by $\chi^{\bm{\lambda}}$.

For every $f\in\Phi$, let $X_f:=\{X_{1,f}, X_{2,f},\ldots\}$ be a
set of infinitely many variables. Each $X_{i,f}$ has degree $d(f)$.

Let
\begin{equation*}
\tilde P_{\eta}(f) =  \tilde{P}_{\eta}(X_{f}; q^{d(f)}) =
q^{-d(f)n(\eta)} P_{\eta}(X_{f}; q^{-d(f)})
\end{equation*}
where $\tilde{P}_{\eta}(X_{f}; q^{d(f)})$ is the twisted
Hall-Littlewood symmetric function. Define
\begin{equation*}
\tilde P_{\bm{\mu}} = \prod_{f\in\Phi} \tilde P_{\bm{\mu}(f)}(f).
\end{equation*}

For every $\varphi\in\Theta$, let $Y^{\varphi}:=\{Y_{1}^{\varphi},
Y_{2}^{\varphi},\ldots\}$ be a set of infinitely many variables.
Each $Y_{i}^{\varphi}$ has degree $d(\varphi)$. Define
\begin{equation*}
S_{\bm{\lambda}} =
\prod_{\varphi\in\Theta}s_{\bm{\lambda}(\varphi)}(Y^{\varphi}),
\end{equation*}
where $s_{\bm{\lambda}(\varphi)}(Y^{\varphi})$ is the Schur
symmetric function.

Let
\begin{equation*}
\Lambda_{\mathbb C} = \otimes_{f\in\Phi} \Lambda_{\mathbb C}(X_f)=
\otimes_{\varphi\in\Theta} \Lambda_{\mathbb C}(Y^{\varphi})
\end{equation*}
where $\Lambda_{\mathbb C}(X_f)$ is the ring of symmetric functions
in $X_{f}$, and $\Lambda_{\mathbb C}(Y^{\varphi})$ is the ring of
symmetric functions in $Y^{\varphi}$. As a graded ring, we have
\begin{align*}
\Lambda_{\mathbb C} &={\mathbb C}\text{-span}\{\tilde
P_{\bm{\mu}} | \bm{\mu}\in\mathcal{P}^{\Phi} \} \\
&={\mathbb C}\text{-span}\{S_{\bm{\lambda}} |
\mathbf{\bm{\lambda}}\in\mathcal{P}^{\Theta} \}
\end{align*}

The transformation between the symmetric functions in the variables
$\{X_{f}: f\in \Phi\}$ and those in the variables $\{Y^{\varphi}:
\varphi\in\Theta\}$ is given by the following identity:
\begin{equation}\label{eq:transform}
p_k(Y^{\varphi})=(-1)^{n-1}\sum_{x\in M_{n}}\xi
(x)p_{n/d(f_x)}(X_{f_x})\,,
\end{equation}
where $\xi\in\varphi$, $x\in f_x$ and $n= k\cdot d(\varphi)$. Here
$p_k(Y^{\varphi})$ and $p_{n/d(f_x)}(X_{f_x})$ are power-sum
symmetric functions.

From \cite{Mac} we know that the conjugacy classes $K^{\bm{\mu}}$ of
$G_n$ are parameterized by $\bm{\mu}\in\mathcal{P}_n^{\Phi}$, and
the irreducible characters $\chi^{\bm{\lambda}}$ of $G_n$ are
indexed by $\bm{\lambda}\in\mathcal{P}_n^{\Theta}$. The following
theorem gives the characteristic map of $G_n$.

\begin{theorem}{\rm(Green~\cite{Nat}, Macdonald~\cite{Mac},
Zelevinski~\cite{Zele})}
Let $A_n$ denote the space of
complex-valued class functions on $G_n$ and $A = \oplus_{n\geq 0}
A_n$. The linear map
\begin{align*}
ch : A &\longrightarrow \Lambda_{\mathbb C}\\
    \chi^{\bm{\lambda}} &\mapsto S_{\bm{\lambda}}, \ \ \ \text{for}\ \bm{\lambda} \in
\mathcal{P}^{\Theta},\\
\pi_{\bm{\mu}} &\mapsto \tilde{P}_{\bm{\mu}}, \ \ \ \text{for}\
\bm{\mu}\in \mathcal{P}^{\Phi},
\end{align*}
is a Hopf algebra isomorphism.
\end{theorem}

\subsection{ Supercharacter theory}
Let $U_n$ be the group of unipotent upper-triangular matrices with
entries in the finite field $\mathbb F_q$ and ones on the diagonal.
This group is the subgroup of the finite general linear group $G_n$.
Although the character theory on $U_n$ is a wild problem, another
slightly coarse version called superclass and supercharacter theory
(Andr\'e~\cite{An95}, Yan~\cite{Yan}) makes it easier to study and
compute. Superclasses are certain unions of conjugacy classes and
supercharacters are sums of irreducible characters. They are
compatible in the sense that supercharacters are constant on
superclasses. The supercharacter theory has a rich combinatorial
structure (ref.~\cite{Nat2}) and connects to some other algebra
structures as well (ref.~\cite{MulAu}).

The superclasses of $U_n$ can be indexed by the $\mathbb
F_q^{\times}$-labeled set partitions, and a supercharacter becomes
an irreducible character if the corresponding indexed $\mathbb
F_q^{\times}$-labeled set partition has no crossing arcs. For the
strict definitions and more details on supercharacters please see
~\cite{Nat2} or ~\cite{MulAu}.

In this paper we consider the linear supercharacters of $U_n$
indexed by
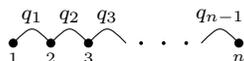
\begin{align*}
\begin{tikzpicture}
\filldraw [black] (0,0) circle (1.5pt) (.5,0) circle (1.5pt) (1,0)
circle (1.5pt) (1.7,0) circle (0.5pt) (2,0) circle (0.5pt) (2.3,0)
circle (0.5pt) (3,0) circle (1.5pt); \node at (0, -.2) {$\scscs 1$};
\node at (0.25, .35) {$\scs q_1$};\node at (0.5, -.2) {$\scscs 2$};
\node at (0.75, .35) {$\scs q_2$}; \node at (1, -.2) {$\scscs 3$};
\node at (1.25, .35) {$\scs q_3$}; \node at (3, -.2) {$\scscs n$};
\node at (2.75, .35) {$\scs q_{n-1}$}; \draw (0,0) .. controls
(.25,.25) .. (.5,0); \draw (.5,0) .. controls (.75,.25) .. (1,0);
\draw (1,0) .. controls (1.25,.25) .. (1.5,0); \draw (2.5,0) ..
controls (2.75,.25) .. (3,0);
\end{tikzpicture}
\end{align*}
where $q_1,\ldots,q_{n-1}\in\mathbb F_q^{\times}$
(Thiem~\cite{Nat2}). Let $\chi^{(n)}_{(q_1,\ldots,q_{n-1})}$ denote
the above character. We induce $\chi^{(n)}_{(q_1,\ldots,q_{n-1})}$
from $U_n$ to $G_n$ by the formula
\begin{equation}\label{eq:inform}
\chi^{(n)}_{(q_1,\ldots,q_{n-1})}\uparrow_{U_n}^{G_n}(g) =
\frac{1}{|U_n|}\sum_{h\in G_n}
\bar{\chi}^{(n)}_{(q_1,\ldots,q_{n-1})} (h g h^{-1}),
\end{equation}
where $\bar{\chi}(s)=\chi(s)$ if $s\in U_n$ and $\bar{\chi}(s)=0$ if
$s\not\in U_n$.

The induced character
$\chi^{(n)}_{(q_1,\ldots,q_{n-1})}\uparrow_{U_n}^{G_n}$ is a
character of $G_n$, which is known as the character of Gelfand-Greav
module. Apply plethysms on the characteristic map of
$\chi^{(n)}_{(q_1,\ldots,q_{n-1})}\uparrow_{U_n}^{G_n}$ and we get a
multiple of the twisted Hall-Littlewood symmetric function
$\tilde{P}_n$. A formula on this result together with a recurrence
relation is given in Section~\ref{sec:Plethysm}.

\section{Plethysm Formula for the induced character}\label{sec:Plethysm}

We start from the formula of the characteristic map of
$\chi^{(n)}_{(q_1,\ldots,q_{n-1})}\uparrow_{U_n}^{G_n}$, which is
given by Thiem~\cite{Nat}.

\begin{theorem} {\rm (Thiem~\cite{Nat})}
\begin{equation}\label{eq:chamap}
ch(\chi^{(n)}_{(q_1,\ldots,q_{n-1})}\uparrow_{U_n}^{G_n}) =
\sum_{\bm{\lambda} \in \mathcal{P}_n^{\Theta}\atop
ht(\bm{\lambda})=1} S_{\bm{\lambda}},
\end{equation}
where $ht(\bm{\lambda})=max\{l(\bm{\lambda} (\varphi)) | \varphi \in
\Theta\}$.
\end{theorem}

Notice that $ht(\bm{\lambda})=1$ implies for every
$\varphi\in\Theta$ we have $l(\bm{\lambda} (\varphi))\leq 1$, which
means $\bm{\lambda} (\varphi)$ contains at most one part. From the
definition of $S_{\bm{\lambda}}$, we can write (\ref{eq:chamap}) as
\begin{align}\label{eq:char}
ch(\chi^{(n)}_{(q_1,\ldots,q_{n-1})}\uparrow_{U_n}^{G_n}) & =
\sum_{\bm{\lambda} \in \mathcal{P}_n^{\Theta}\atop
ht(\bm{\lambda})=1} \prod_{\varphi\in\Theta}
s_{\bm{\lambda}^{(\varphi)}}(Y^{\varphi})\nonumber\\
& =\sum_{a_1b_1+\cdots+a_kb_k=n\atop
\bm{\lambda}(\Theta)=\{a_1,\ldots,a_k\}\in
\mathcal{P}_n^{\Theta}}\sum_{deg(\varphi_i)=b_i\atop
\varphi_1,\ldots,\varphi_k \text{ distinct}}
h_{a_1}(Y^{\varphi_1})h_{a_2}(Y^{\varphi_2})\cdots
h_{a_k}(Y^{\varphi_k}) ,
\end{align}
where $Y^{\varphi_1}, Y^{\varphi_2},\ldots, Y^{\varphi_k}$ are
different sets of variables. For $i$ from $1$ to $k$, each variable
in the set $Y^{\varphi_i}=\{Y^{\varphi_i}_1,Y^{\varphi_i}_2,\ldots
\}$ has degree $b_i$.

We give an example to better understand formula~(\ref{eq:chamap})
and (\ref{eq:char}).
\begin{example}
For $n=3$, we have
\begin{align*}
ch(\chi^{(3)}_{(q_1,q_{2})}\uparrow_{U_3}^{G_3}) &=
\sum_{\varphi_1,\varphi_2,\varphi_3 \text{distinct}\atop
deg(\varphi_i)=1}
h_1(Y^{\varphi_1})h_1(Y^{\varphi_2})h_1(Y^{\varphi_3}) \\
& + \sum_{\psi_1,\psi_2 \text{distinct}\atop deg(\psi_i)=1}
h_2(Y^{\psi_1})h_1(Y^{\psi_2}) + \sum_{deg(\bar{\varphi}_1)=2\atop
deg(\bar{\varphi}_2)=1} h_1(Y^{\bar{\varphi}_1})h_1(Y^{\bar{\varphi}_2})\\
& + \sum_{deg(\varphi)=1} h_3(Y^{\varphi}) + \sum_{deg(\psi)=3}
h_1(Y^{\psi}).
\end{align*}
\end{example}
From the above example we see that the expansion on the right-hand
side of~(\ref{eq:char}) becomes more complicated as $n$ increases.
This inspires us to use plethysm to simplify the computation.

For each term in equation~(\ref{eq:char}), we have a {\sl two-rowed}
array $\left(\begin{array}{cccc}
b_1 & b_2 & \cdots & b_k \\
a_1 & a_2 & \cdots & a_k
\end{array}\right)$ where $b_i=d(\varphi_i)$ and it satisfies the condition $\sum_{i=1}^{k} a_i b_i =
n$. We arrange the pairs $(b_i,a_i)$ such that:

(1) $b_1\leq b_2\leq\ldots\leq b_k$,

(2) $a_j\leq a_{j+1}$ if $b_j=b_{j+1}$ for $1\leq j<k$.

Once the array is sorted, we can denote it as follows:
\begin{equation*}
\left(\begin{array}{ccccccccccccc}
   &  & 1^{m_1} &  &  &  & 2^{m_2} &  & \cdots  &  &  & n^{m_n} &  \\
  1^{m_{1,1}} & 2^{m_{1,2}} & \cdots & n^{m_{1,n}} & 1^{m_{2,1}} & 2^{m_{2,2}} & \cdots & n^{m_{2,n}} & \cdots & 1^{m_{n,1}} & 2^{m_{n,2}} & \cdots & n^{m_{n,n}}
\end{array}\right)
\end{equation*}
where $\sum_{i,j=1}^{n} m_{i,j}\times j\times i=n$ and
$m_{i,1}+m_{i,2}+\ldots+m_{i,n}=m_i$ for $1\leq i \leq n$. Each
$m_i$ counts the number of different sets of variables appearing in
the term with the same degree $i$. Each $m_{i,j}$ counts the number
of complete symmetric functions $h_{j}$ in variables with degree
$i$.

For a given $i$, let $l_q(i)$ denote the number of all different
sets of variables with the same degree $i$. We know that $l_q(i)$ is
equal to the number of irreducible polynomials $f$ over finite field
$\mathbb F_q$ with degree $i$ and satisfying $f(0)\neq 0$. The
number of irreducible polynomials of degree $i$ over $\mathbb F_q$
is given by the fomula
$$L_q(i)=\frac{1}{i}\sum_{d|i}\mu(d)q^{\frac{i}{d}},$$
where $\mu$ is the M\"{o}bius function. Then we have
\begin{equation*}
l_q(i)=\left\{
  \begin{array}{ll}
    L_q(1)-1, & \text{for}\ i=1; \\
    L_q(i), & \text{for}\ i\geq 2.
  \end{array}
\right.
\end{equation*}
Thus for a given $i$ and a list of numbers $(m_{i,1}, m_{i,2},
\ldots, m_{i,n})$ where $m_{i,1}+m_{i,2}+\ldots + m_{i,n}=m_i$, the
number of products in the form
\begin{align}\label{eq:prod}
& h_{1}(Y^{\varphi_{i,1}})h_{1}(Y^{\varphi_{i,2}})\cdots
h_{1}(Y^{\varphi_{i,m_{i,1}}}) \nonumber\\
& \times
h_{2}(Y^{\varphi_{i,m_{i,1}+1}})h_{2}(Y^{\varphi_{i,m_{i,1}+2}})\cdots
h_{2}(Y^{\varphi_{i,m_{i,1}+m_{i,2}}})\nonumber\\
& \times \cdots \nonumber\\
& \times h_{n}(Y^{\varphi_{i,m_{i,1}+\cdots +
m_{i,n-1}+1}})h_{n}(Y^{\varphi_{i,m_{i,1}+\cdots +
m_{i,n-1}+2}})\cdots h_{n}(Y^{\varphi_{i,m_{i}}})
\end{align}
is equal to
$$\frac{l_q(i)(\l_q(i)-1)\cdots(l_q(i)-m_{i}+1)}{m_{i,1}!m_{i,2}!\cdots
m_{i,n}!} , $$ where $Y^{\varphi_{i,1}}, Y^{\varphi_{i,2}}, \ldots,
Y^{\varphi_{i,m_{i}}}$ are $m_i$ different sets of variables with
the same degree $i$. Notice that when $n$ increases, we get more
terms on the right-hand side of equation~(\ref{eq:char}).

In order to simplify the computation, we apply plethysms
on~(\ref{eq:char}) which means replacing each set of variables
$Y^{\varphi_i}$ by $\{y^{b_i}_1, y^{b_i}_2, \ldots\}$. In doing so
we don't differentiate the sets of variables. It seems that we lose
information by applying the plethysms on the characteristic map.
However this is not the case as we see later on in
Theorem~\ref{thm:comm} and Corollary~\ref{col:equi}.

\begin{definition}
Define the {\sl plethysm map} $\rho: {\mathbb
C}\text{-span}\{S_{\bm{\lambda}} |
\mathbf{\bm{\lambda}}\in\mathcal{P}^{\Theta} \} \rightarrow
\Lambda_{\mathbb C}(Y) $ as follows:
$$\rho(h_{a}(Y^{\varphi})) = h_{a}[p_{b}(Y)], \ \forall \varphi\in\Theta, b=deg(\varphi).$$
\end{definition}
Since $ch(\chi^{(n)}_{(q_1,\ldots,q_{n-1})}\uparrow_{U_n}^{G_n})$ is
independent from $q_1,\ldots,q_{n-1}$, for $n\geq 1$ we simply
denote
$\rho(ch(\chi^{(n)}_{(q_1,\ldots,q_{n-1})}\uparrow_{U_n}^{G_n}))$ by
$\rho_n$ and set $\rho_0=1$. We also use $\rho_{(m_{i,1},\ldots,
m_{i,n})}$ to denote the results of taking plethysms on the sum of
all different products in the form of (\ref{eq:prod}) for the same
index list $(m_{i,1},\ldots, m_{i,n})$, i.e.
$$\rho_{(m_{i,1},\ldots, m_{i,n})}:=\frac{l_q(i)(\l_q(i)-1)\cdots(l_q(i)-m_{i}+1)}{m_{i,1}!m_{i,2}!\cdots
m_{i,n}!}(h_{1}[p_{i}])^{m_{i,1}}\cdots (h_{n}[p_{i}])^{m_{i,n}}.$$
Taking plethysms on both sides of (\ref{eq:char}) we get
\begin{equation}\label{eq:plechar}
\rho_n = \sum_{\sum_{i,j=1}^{n} m_{i,j}\times j\times i=n}
\rho_{(m_{1,1},\ldots, m_{1,n})}\cdots \rho_{(m_{n,1},\ldots,
m_{n,n})} .
\end{equation}

The following theorem falls naturally.
\begin{theorem}\label{thm:lm1}
Let $CH(t)$ denote the generating function for $\rho_n$ as follows:
$$CH(t) = 1+ \rho_1 t+\rho_2 t^2+\cdots = \sum_{s\geq 0} \rho_n t^n . $$
Then we have
$$CH(t) = \prod_{i\geq 1}\left(\prod_{j\geq 1}(1-y_j^{i}t^{i})^{-1}\right)^{lq(i)} = \prod_{i\geq 1}\prod_{j\geq 1}(1-y_j^{i}t^{i})^{-lq(i)}.$$
\end{theorem}
\begin{proof}
Since for every $i\geq 1$,
\begin{align*}
\prod_{j\geq 1}(1-y_j^{i}t^{i})^{-1} & = \sum_{a\geq 0} h_a(y_1^i,
y_2^i, \ldots)t^{a\cdot i}\\
& = 1+ (h_{1}[p_{i}])\cdot t^{i}+(h_{2}[p_{i}])\cdot t^{2i} + \cdots
.
\end{align*}
We have
\begin{align*}
&\left(\prod_{j\geq 1} (1-y_j^{i}t^{i})^{-1}\right)^{lq(i)} \\
&\ \ \ \  = \left( 1+ (h_{1}[p_{i}]) t^{i}+(h_{2}[p_{i}]) t^{2i} +
\cdots\right)^{lq(i)}\\
&\ \ \ \  = \sum_{m_{i,1}+m_{i,2}+\cdots+m_{i,n}=m_{i}\atop 0\leq
m_i \leq lq(i)} {lq(i)\choose
m_{i}}{m_{i}\choose m_{i,1} m_{i,2} \cdots m_{i,n}}\\
&\ \ \ \ \qquad\qquad\quad \times
(h_{1}[p_{i}])^{m_{i,1}}(h_{2}[p_{i}])^{m_{i,2}}\cdots(h_{n}[p_{i}])^{m_{i,n}}\cdot
t^{(m_{i,1}+2m_{i,2}+\cdots+n\cdot
m_{i,n})\cdot i}\\
&\ \ \ \ = \sum_{m_{i,1}+m_{i,2}+\cdots+m_{i,n}=m_{i}\atop 0\leq m_i
\leq lq(i)} \rho_{(m_{i,1},\ldots, m_{i,n})}\cdot
t^{(m_{i,1}+2m_{i,2}+\cdots+n\cdot m_{i,n})\cdot i}
\end{align*}
From (\ref{eq:plechar}) we see that the coefficient of $t^n$ in the
product $\prod_{i\geq 1}\left(\prod_{j\geq
1}(1-y_j^{i}t^{i})^{-1}\right)^{lq(i)}$ is exactly equal to $\rho_n$
for $n\geq 1$. Thus we get the Theorem.
\end{proof}

\begin{theorem}\label{thm:lm2}
\begin{equation}\label{eq:lqt}
\prod_{i\geq 1}\prod_{j\geq 1}(1-y_j^{i}t^{i})^{-lq(i)} =
\frac{\prod_{j\geq 1}(1-y_j q t)^{-1}}{\prod_{j\geq 1}(1-y_j
t)^{-1}} \lower 15pt\hbox{.}
\end{equation}
\end{theorem}
\begin{proof}
The above identity is equivalent to the identity
\begin{equation}\label{eq:qt}
\prod_{i\geq 1}\prod_{j\geq 1}(1-y_j^{i}t^{i})^{Lq(i)} =
\prod_{j\geq 1}(1-y_j q t),
\end{equation}
where $Lq(i)$ denotes the number of irreducible polynomials over
$\mathbb F_q$ for $i\geq 1$ as we stated before. To prove
(\ref{eq:qt}), we take the logarithm on both sides of (\ref{eq:qt})
and show they are equal.
\begin{align*}
\ln\left(\prod_{i\geq 1}\prod_{j\geq
1}(1-y_j^{i}t^{i})^{Lq(i)}\right) &= \sum_{j\geq 1}\left(\sum_{i\geq
1}
Lq(i)\ln(1-y_j^{i}t^{i})\right)\\
&= \sum_{j\geq 1}\left(\sum_{i\geq 1} Lq(i)\left(\sum_{r\geq 1}
\frac{(y_j^i t^i)^{r}}{r} \right)\right)\\
&= \sum_{j\geq 1}\left(\sum_{i\geq 1} \sum_{r\geq 1}
Lq\left(\frac{i\cdot r}{r}\right)\cdot \frac{i\cdot r}{r}\cdot
\frac{y_j^{(i\cdot r)}\cdot t^{(i\cdot r)}}{i\cdot r} \right)\\
&=\sum_{j\geq 1}\left(\sum_{N\geq 1\atop N=i\cdot r}\frac{
y_j^{(N)}\cdot t^{(N)}}{N}\left( \sum_{r|N}
Lq\left(\frac{N}{r}\right)\cdot \frac{N}{r} \right)\right)\\
&=\sum_{j\geq 1}\left(\sum_{N\geq 1\atop N=i\cdot r}\frac{
y_j^{(N)}\cdot t^{(N)}}{N}\cdot q^{N}\right)\\
&=\sum_{j\geq 1}\left(\ln(1-y_j q t)\right) = \ln\left(\prod_{j\geq
1}(1-y_j q t)\right)
\end{align*}
Thus we get (\ref{eq:qt}).
\end{proof}
Theorem~(\ref{thm:lm1}) and Theorem~(\ref{thm:lm2}) together yield
the formula for the generating function of $\rho_n$ as follows:
\begin{equation}\label{eq:gen}
CH(t) = \prod_{j\geq 1}\frac{1-y_j t}{1-y_j q t}\, \lower
8pt\hbox{.}
\end{equation}
Before we link it to Hall-Littlewood polynomials, we give a
recurrence relation for $\rho_n$ using formula~(\ref{eq:gen}).

\begin{corollary}\label{col:thmrec}
For every $n\geq 1$, we have
\begin{equation}\label{eq:rec}
\rho_n = (q^n-1) h_{n} - \rho_{n-1} h_{1} - \rho_{n-2} h_{2} \cdots
- \rho_1 h_{n-1} .
\end{equation}
\end{corollary}
\begin{proof}
From~(\ref{eq:gen}) we have
\begin{equation*}
CH(t)\times H(t) = \prod_{j\geq 1}(1-y_j q t)^{-1} .
\end{equation*}
Compare the coefficients of $t^n$ on both sides we get
\begin{equation*}
\rho_0 h_{n}+\rho_1 h_{n-1}+\cdots +\rho_n h_{0} = q^n h_{n},
\end{equation*}
which yields the theorem.
\end{proof}

\begin{example}
\begin{align*}
\rho_1 &= (q-1) h_1 ;\\
\rho_2 &= (q^2-1) h_2 - \rho_1 h_{1} \\
        &= (q^2-1) h_2 - (q-1) h_{1,1}\\
        &= (q-1)[(q+1) h_2 - h_{1,1}] ;\\
\rho_3 &= (q^3-1) h_3 - \rho_1 h_{2} - \rho_2 h_{1}\\
        &= (q^3-1) h_3 - (q-1) h_{2,1} - (q^2-1) h_{2,1} + (q-1) h_{1,1,1}\\
        &= (q-1)[(q^2+q+1) h_3 -(q+2) h_{2,1} + h_{1,1,1}] .\\
\end{align*}
\end{example}
From the above examples we notice that the coefficients of
$h_{\lambda}$ are in $\pm\mathbb N[q]\times (q-1)$.  Let
$[h_{\lambda}]\rho_n$ denote the coefficients of $h_{\lambda}$ in
the expansion of $\rho_n$. In particular we have $[h_{n}]\rho_n =
q^n-1$ for all $n\geq 1$. The following corollary gives the
recurrence relation on the coefficients.

\begin{corollary}\label{col:rec}
For any $\lambda = (a_1^{l_1}, a_2^{l_2}, \ldots, a_k^{l_k}) \vdash
n$ with $l_i \geq 1$ for all $1\leq i \leq k$ and $l(\lambda)\geq
2$, we have
\begin{equation}\label{eq:reccoeff}
[h_{\lambda}]\rho_n = -[h_{(a_1^{l_1-1}, a_2^{l_2}, \ldots,
a_k^{l_k})}]\rho_{n-a_1} - \cdots -[h_{(a_1^{l_1}, a_2^{l_2},
\ldots, a_k^{l_k-1})}]\rho_{n-a_k}.
\end{equation}
Here if $l_i=1$ for some $1\leq i\leq k$, then we set
$$(a_1^{l_1}, \ldots, a_i^{l_i-1}, \ldots, a_k^{l_k}):= (a_1^{l_1}, \ldots, \hat{a_i}, \ldots, a_k^{l_k}) , $$
where $\hat{a_i}$ means simply remove $a_i$ from the partition
$\lambda$. In particular, $[h_{\lambda}]\rho_n\in \pm\mathbb
N[q]\times(q-1)$ while the sign is given by $(-1)^{l(\lambda)-1}$.
\end{corollary}
\begin{proof}
Equation~(\ref{eq:reccoeff}) follows directly from
Corollary~\ref{col:thmrec} by comparing the coefficients of
$h_{\lambda}$ from two sides. The claim that $[h_{\lambda}]\rho_n$
is in $\pm\mathbb N[q]\times(q-1)$ together with the sign property
can be proved easily by using induction method on
equation~(\ref{eq:reccoeff}).
\end{proof}

\begin{remark}
Corollary~\ref{col:thmrec} and Corollary~\ref{col:rec} give an easy
way of computing $\rho_n$ for every $n\geq 1$ simply by knowing
$[h_i]\rho_i=q^i-1$ for every $i\geq 1$.
\end{remark}

\begin{example}
\begin{align*}
[h_{2,1}]\rho_3& = -[h_{1}]\rho_1 - [h_{2}]\rho_2\\
                & = - (q-1) - (q^2-1) \\
                & = -(q-1)(q+2)
\end{align*}
\begin{align*}
[h_{1,1,1}]\rho_3& = -[h_{1,1}]\rho_2 =[h_{1}]\rho_1\\
                  & = q-1
\end{align*}
\end{example}

Now back to our formula~(\ref{eq:gen}). We rewrite it into the
following form so that we can easily use the generating function for
$q_r$ as in equation~(\ref{eq:genqr}).
\begin{align*}
CH(t) &= \prod_{j\geq 1}\frac{1-y_j t}{1-y_j q t} = \prod_{j\geq
1}\frac{1-y_j\cdot \frac{1}{q}\cdot (q t)}{1-y_j\cdot (q t)} \\
      &=\sum_{r\geq 0} q_r(Y;q^{-1}) q^{r} t^{r}  \lower
5pt\hbox{,}
\end{align*}
where $Y=\{y_1, y_2, \ldots\}$. Comparing the coefficients from two
sides we get the following corollary.
\begin{corollary}\label{col:halli}
\begin{align}\label{eq:HL}
\rho_n &=  q_n(Y;q^{-1}) q^{n} = (1-q^{-1})P_n\left(Y;q^{-1}\right) q^{n}\nonumber\\
        &= q^{n-1}(q-1) P_n\left(Y;q^{-1}\right)= q^{n-1}(q-1) \tilde{P}_n\left(Y;q\right).
\end{align}
\end{corollary}
Corollary~\ref{col:halli} gives the connection between the plethysm
of the characteristic map of
$\chi^{(n)}_{(q_1,\ldots,q_{n-1})}\uparrow_{U_n}^{G_n}$ and the
Hall-littlewood symmetric functions.

For any linear supercharacter~\cite{Nat2, Nat, MulAu} of $U_n$,
there is a unique way to decompose the indexed set partition into
connected components. For a linear supercharacter with $k$ connected
components, we can denote it by $\chi^{n_1 | n_2 | \ldots |
n_k}_{\vec{q_1},\ldots,\vec{q_{k}}}$ where for $i$ from $1$ to $k$,
each $n_i$ counts the size of the $i^{th}$ connected component and
$\vec{q_i}=(q_{i,1}, \ldots, q_{i,n_i-1})\in (\mathbb
F_q^{\times})^{n_i-1}$ denotes the labels of the arcs for the
$i^{th}$ connected component. The following corollary follows from
the property of the linear supercharacters~\cite{Nat2, Nat, MulAu}.
\begin{corollary}\label{col:multilichar}
$$\rho\circ ch(\chi^{n_1 | n_2 | \ldots |
n_k}_{\vec{q_1},\ldots,\vec{q_{k}}}\uparrow_{U_n}^{G_n}) =
\prod_{i=1}^k \rho_{n_i}.$$
\end{corollary}

\begin{example} For the following linear supercharacter of $U_6$
\begin{align*}
\raise 2pt \hbox{$\chi^{1 | 2 |
3}_{\vec{q_1},\vec{q_2},\vec{q_{3}}}$ = } \chi^{\begin{tikzpicture}
\filldraw [black] (0,0) circle (1.5pt) (.5,0) circle (1.5pt) (1,0)
circle (1.5pt) (1.5,0) circle (1.5pt) (2,0) circle (1.5pt) (2.5,0)
circle (1.5pt); \node at (0, -.2) {$\scscs 1$}; \node at (0.5, -.2)
{$\scscs 2$}; \node at (0.75, .35) {$\scs q_{2,1}$}; \node at (1,
-.2) {$\scscs 3$}; \node at (1.5, -.2) {$\scscs 4$}; \node at (1.75,
.35) {$\scs q_{3,1}$};\node at (2, -.2) {$\scscs 5$}; \node at
(2.25, .35) {$\scs q_{3,2}$};\node at (2, -.2) {$\scscs 5$}; \node
at (2.5, -.2) {$\scscs 6$}; \draw (0.5,0) .. controls (.75,.25) ..
(1,0); \draw (1.5,0) .. controls (1.75,.25) .. (2,0); \draw (2,0) ..
controls (2.25,.25) .. (2.5,0);
\end{tikzpicture}}
\end{align*}
where $\vec{q_1}=0, \vec{q_2}=(q_{2,1}),
\vec{q_{3}}=(q_{3,1},q_{3,2})$ and $q_{2,1}, q_{3,1},
q_{3,2}\in\mathbb F_q^{\times}$, we have
$$\rho\circ ch(\chi^{1 | 2 |
3}_{\vec{q_1},\vec{q_2},\vec{q_{3}}}\uparrow_{U_6}^{G_6}) =
\rho_1\rho_2\rho_3.$$
\end{example}

Let the transition matrix between $\{m_{\lambda}(X)\}_{\lambda\vdash
n} $ and  $\{p_{\mu}(X)\}_{\mu\vdash n} $ be $C_{\lambda, \mu}$,
i.e.
\begin{equation*}
m_{\lambda}(X)=\sum_{\mu} C_{\lambda, \mu} p_{\mu}(X).
\end{equation*}
Define $m_{\lambda}(q-1)$ by the following equation
\begin{equation*}
m_{\lambda}(q-1)=\sum_{\mu} C_{\lambda, \mu} p_{\mu}(q-1),
\end{equation*}
where $p_{n}(q-1) = q^n-1$ for every $n\geq 1$, and $p_{\mu}(q-1) =
p_{\mu_1}(q-1)\cdots p_{\mu_l}(q-1)$ for
$\mu=\{\mu_1,\ldots,\mu_l\}$.

\begin{remark}
Using the orthogonal relation between the bases $\{m_{\lambda}\}$
and $\{h_{\mu}\}$, we give another expression for $\rho_n$ as
follows:
\begin{equation*}
\rho_n = \sum_{\lambda\vdash n}m_{\lambda}(q-1)\cdot h_{\lambda}(Y).
\end{equation*}
\end{remark}
\begin{proof}
Using the notation in Section~\ref{eq:omg}, we have
\begin{equation*}
\Omega(Yqt) = \prod_{j\geq 1}\frac{1}{1-y_j q t} , \ \  \Omega(-Yt)
= \prod_{j\geq 1}(1-y_j t).
\end{equation*}
\begin{align*}
CH(t) &= \prod_{j\geq 1}\frac{1-y_j t}{1-y_j q t} \\
      &= \Omega [(q-1)Yt]\\
      &= \sum_{n\geq 0}\left(\sum_{\lambda\vdash n}m_{\lambda}(q-1)\cdot h_{\lambda}(Y)
\right) t^n \, \lower 8pt\hbox{.}
\end{align*}
\end{proof}

It seems that we lose much information by taking plethysms on the
characteristic map of
$\chi^{(n)}_{(q_1,\ldots,q_{n-1})}\uparrow_{U_n}^{G_n}$. However if
we only consider the induced characters from $U_n$ to $G_n$, we can
express the characteristic map of the induced characters in basis
$\{\tilde P_{\bm{\mu}} | \bm{\mu}\in\mathcal{P}^{\Phi} \}$ from the
results of doing plethysms. To show this fact, we first introduce
the following homomorphism defined in~\cite{Mac}:
$$\omega: \Lambda_{\mathbb C}(Y) \rightarrow \Lambda_{\mathbb C}(Y)$$
by
$$\omega(e_r(Y))=h_r(Y),\ \text{for\ all}\ r\geq 0.$$

\begin{lemma}{\rm (\cite{Mac})}\label{lem:involution}
$\omega$ is an involution and automorphism on $\Lambda_{\mathbb
C}(Y)$. Also, we have
$$\omega(p_r(Y))=(-1)^{r-1}p_r(Y),\ \text{for\ all}\ r\geq 0.$$
\end{lemma}

The following theorem illustrates the relation between the
application of plethysms on the characteristic map in basis
$\{S_{\bm{\lambda}} | \mathbf{\bm{\lambda}}\in\mathcal{P}^{\Theta}
\}$ and the characteristic map in basis $\{\tilde P_{\bm{\mu}} |
\bm{\mu}\in\mathcal{P}^{\Phi} \}$.
\begin{theorem}\label{thm:comm}
The following diagram commutes:
\input xy
\xyoption{all}
\begin{displaymath}
\xymatrix{
  \otimes_{\varphi\in\Theta} \Lambda_{\mathbb C}(Y^{\varphi}) \ar[d]_{\rho} \ar[r]^{T}
                & \otimes_{f\in\Phi} \Lambda_{\mathbb C}(X_f) \ar[d]^{\Pi |_{\Lambda_{\mathbb C}(X_{f=x-1})}}  \\
 \Lambda_{\mathbb C}(Y)           \ar[r]_{t\circ\omega}
                & \Lambda_{\mathbb C}(X_{x-1})}
\end{displaymath}
where $T$ is the map of transformation from basis
$\{S_{\bm{\lambda}} | \mathbf{\bm{\lambda}}\in\mathcal{P}^{\Theta}
\}$ to basis $\{\tilde P_{\bm{\mu}} | \bm{\mu}\in\mathcal{P}^{\Phi}
\}$, $t$ is the map of changing variables $y_i$ into $X_{i,x-1}$ for
$i=1,2,\ldots$, and $\Pi |_{\Lambda_{\mathbb C}(X_{f=x-1})}$ is the
projection to the space $\Lambda_{\mathbb C}(X_{x-1})$.
\end{theorem}
\begin{proof}
We rewrite equation~(\ref{eq:transform}) as follows
\begin{equation*}
p_k(Y^{\varphi})=(-1)^{n-1}\sum_{x\in M_{n}}\xi
(x)p_{n/d(f_x)}(X_{f_x})\,,
\end{equation*}
where $\xi\in\varphi$, $x\in f_x$ and $n= k\cdot d(\varphi)$. If we
apply plethysm on $p_k(Y^{\varphi})$ we get $p_n(Y)$. Applying the
projection map $\Pi |_{\Lambda_{\mathbb C}(X_{f=x-1})}$ on the
right-hand side of equation~(\ref{eq:transform}) yields
$(-1)^{n-1}p_n(X_{x-1})$. Since $\{p_n: n=1,2,\ldots\}$ are
algebraically independent over $\mathbb C$ and $\{p_{\lambda}:
\lambda\ \text{a\ partition}\}$ form a basis for $\Lambda_{\mathbb
C}$, we get the theorem from Lemma~\ref{lem:involution}.
\end{proof}

\begin{corollary}\label{col:equi}
If we use
$ch(\chi^{(n)}_{(q_1,\ldots,q_{n-1})}\uparrow_{U_n}^{G_n})(X_f: f\in
\Phi)$ to denote the expression of the characteristic map of
$ch(\chi^{(n)}_{(q_1,\ldots,q_{n-1})}\uparrow_{U_n}^{G_n})$ in terms
of basis  $\{\tilde P_{\bm{\mu}} | \bm{\mu}\in\mathcal{P}^{\Phi}\}$,
then we have the following identity:
\begin{align*}
ch(\chi^{(n)}_{(q_1,\ldots,q_{n-1})}\uparrow_{U_n}^{G_n})(X_f: f\in
\Phi)&=t\circ\omega (\rho_n)\\
&=q^{n-1}(q-1) \omega(\tilde{P}_n\left(X_{x-1}\right)).
\end{align*}
\end{corollary}
\begin{proof}
From the definition of the induced character by
equation~(\ref{eq:inform}) we know that
$$\chi^{(n)}_{(q_1,\ldots,q_{n-1})}\uparrow_{U_n}^{G_n}(g)=0$$ for
all $g\in G_n$ which are not similar to any unipotent
upper-triangular matrices. Notice that the characteristic polynomial
for all matrices in $U_n$ is $(x-1)^n$. Since similar matrices have
the same characteristic polynomial,
$\chi^{(n)}_{(q_1,\ldots,q_{n-1})}\uparrow_{U_n}^{G_n}$ could
possibly take nonzero values only on those matrices in $G_n$ with
characteristic polynomials equal to $(x-1)^n$. We then have
$$ch(\chi^{(n)}_{(q_1,\ldots,q_{n-1})}\uparrow_{U_n}^{G_n})(X_f: f\in
\Phi) \in \Lambda_{\mathbb C}(X_{x-1})$$ and so
$$\Pi |_{\Lambda_{\mathbb C}(X_{f=x-1})}[ch(\chi^{(n)}_{(q_1,\ldots,q_{n-1})}\uparrow_{U_n}^{G_n})]
=ch(\chi^{(n)}_{(q_1,\ldots,q_{n-1})}\uparrow_{U_n}^{G_n}).$$ By
theorem~\ref{thm:comm} we obtain the corollary.
\end{proof}
\begin{remark}\label{rm:equi}
From the proof of Corollary~\ref{col:equi} we conclude that for any
character $\chi$ of $U_n$, if we induce $\chi$ from $U_n$ to $G_n$,
then we have
\begin{equation*}
ch(\chi\uparrow_{U_n}^{G_n})(X_f: f\in \Phi) = t\circ\omega\circ\rho
(ch(\chi\uparrow_{U_n}^{G_n})(Y^{\varphi}:\varphi\in\Theta)).
\end{equation*}
\end{remark}

For $\lambda={\lambda_1,\ldots,\lambda_l}$ let
$\rho_{\lambda}=\rho_{\lambda_1}\rho_{\lambda_1}\ldots\rho_{\lambda_l}$.
By Corollary~\ref{col:halli} since $\rho_n = q^{n-1}(q-1)
P_n(Y;q^{-1})$ we know that $\{\rho_{\lambda}\}$ forms a basis for
the symmetric function ring $\Lambda_{\mathbb C}(Y)$. Thus
$\rho(ch(\chi\uparrow_{U_n}^{G_n}))$ can be written into
$\rho(ch(\chi\uparrow_{U_n}^{G_n})) = \sum_{\lambda\vdash n}
C_{\lambda}\rho_{\lambda}$ where $C_{\lambda}\in\mathbb C$. We then
define a map as follows.
\begin{definition}
Define $\hat{\rho}: \Lambda_{\mathbb C}(Y) \rightarrow {\mathbb
C}\text{-span}\{S_{\bm{\lambda}} |
\mathbf{\bm{\lambda}}\in\mathcal{P}^{\Theta} \}$ by
\begin{align*}
\hat{\rho}(\rho_n)&:= \sum_{\bm{\lambda} \in
\mathcal{P}_n^{\Theta}\atop ht(\bm{\lambda})=1}
S_{\bm{\lambda}}\\
&=ch(\chi^{(n)}_{(q_1,\ldots,q_{n-1})}\uparrow_{U_n}^{G_n}) .
\end{align*}
and
\begin{align*}
\hat{\rho}(\rho_{\lambda})=\hat{\rho}(\rho_{\lambda_1})\hat{\rho}(\rho_{\lambda_2})\ldots\hat{\rho}(\rho_{\lambda_l}),
\end{align*}
where $\lambda=(\lambda_1,\ldots,\lambda_l)$.
\end{definition}

\begin{proposition}\label{prop:conjecture2}
For a fixed finite field $\mathbb F_q$ and a character $\chi$ of
$U_n$, we have
\begin{equation*}
(\hat{\rho}\circ\rho) (ch(\chi\uparrow_{U_n}^{G_n})) =
ch(\chi\uparrow_{U_n}^{G_n}).
\end{equation*}
\end{proposition}
\begin{proof}
Since $\omega$ is an automorphism and $\rho$ is multiplicative, the
proposition follows from Theorem~\ref{thm:comm} and
Remark~\ref{rm:equi}.
\end{proof}
Suppose $\rho(ch(\chi\uparrow_{U_n}^{G_n})) = \sum_{\lambda\vdash n}
C_{\lambda}\rho_{\lambda}$ where $C_{\lambda}\in\mathbb C$, from the
definition of $\hat{\rho}$ we get
\begin{align}\label{eq:inverse}
\hat{\rho}\circ\rho(ch(\chi\uparrow_{U_n}^{G_n})) &=
\sum_{\lambda\vdash n} C_{\lambda}(\hat{\rho}(\rho_{\lambda})) \nonumber\\
&=\sum_{\lambda\vdash n}
C_{\lambda}\hat{\rho}(\rho_{\lambda_1})\hat{\rho}(\rho_{\lambda_2})\ldots\hat{\rho}(\rho_{\lambda_l}).
\end{align}
Using Proposition~\ref{prop:conjecture2} we get the following
corollary.
\begin{corollary}\label{col:irrdecomp}
For a fixed finite field $\mathbb F_q$ and a character $\chi$ of
$U_n$, suppose $\nobreak{ch(\chi\uparrow_{U_n}^{G_n})} =
\sum_{\lambda\vdash n} C_{\lambda}\rho_{\lambda}$ where
$C_{\lambda}\in\mathbb C$. We have
\begin{align*}
ch(\chi\uparrow_{U_n}^{G_n})&= \sum_{\lambda\vdash n}
C_{\lambda}\rho_{\lambda}\\
&=\sum_{\lambda\vdash n} C_{\lambda}\left(\sum_{\bm{\lambda}^{(1)}
\in \mathcal{P}_{\lambda_1}^{\Theta}\atop ht(\bm{\lambda}^{(1)})=1}
S_{\bm{\lambda}^{(1)}}\right) \left(\sum_{\bm{\lambda}^{(2)} \in
\mathcal{P}_{\lambda_2}^{\Theta}\atop ht(\bm{\lambda}^{(2)})=1}
S_{\bm{\lambda}^{(2)}}\right)\cdots\left(\sum_{\bm{\lambda}^{(l)}
\in \mathcal{P}_{\lambda_l}^{\Theta}\atop ht(\bm{\lambda}^{(l)})=1}
S_{\bm{\lambda}^{(l)}}\right) \lower 25pt \hbox{.}
\end{align*}
\end{corollary}

\begin{remark}
It is difficult to get an expression for
$ch(\chi\uparrow_{U_n}^{G_n})$ in terms of basis $\{S_{\bm{\lambda}}
| \mathbf{\bm{\lambda}}\in\mathcal{P}^{\Theta} \}$, which gives the
irreducible decomposition of the induced character. However if we
know the plethysm of the characteristic map of
$\chi\uparrow_{U_n}^{G_n}$, we may use $\hat{\rho}$ to get the
irreducible decomposition of $ch(\chi\uparrow_{U_n}^{G_n})$. We hope
the results could contribute to research in this problem and we list
some open problems in Section~\ref{sec:probs}.
\end{remark}

\section{Further Questions}\label{sec:probs}
The induced characters that we are studying in this paper are very
special, so a natural question to ask is if we can give a nice
formula for the characteristics of all the induced supercharacters
from $U_n$ to $G_n$. Zelevinsky~\cite{Zele} and Thiem and
Vinroot~{\cite{NatVin}} have worked on the case of degenerate
Gelfand-Graev characters. The question of how the generalized
Gelfand-Graev representations of the finite unitary group decompose
is still open. The generalized Gelfand-Graev representations, which
are defined by Kawanaka~\cite{Kawa}, are obtained by inducing
certain irreducible representations from a unipotent
subgroup~\cite{NatVin}. Here the supercharacters that we are
considering are more general than the case of the generalized
Gelfand-Graev representations. We hope that the ideas and results
developed in this paper could help to work on this direction.

Let us compute plethysms of the characteristic map of some induced
supercharacters.
\begin{example}\label{ex:example}
For $q=2$, we have
\begin{align*}
\lower 5pt \hbox{$\rho\circ ch$} \left(\lower 4pt
\hbox{$\chi^{\begin{tikzpicture} \filldraw [black] (0,0) circle
(1.5pt) (.5,0) circle (1.5pt)  (1,0) circle (1.5pt); \node at (0,
-.2) {$\scscs 1$}; \node at (0.5, -.2) {$\scscs 2$}; \node at (0.5,
.55) {$\scs 1$}; \node at (1, -.2) {$\scscs 3$}; \draw (0,0) ..
controls (.25,.4) and (.75,.4).. (1,0);
\end{tikzpicture}} \uparrow_{U_3}^{G_3}$} \right) = (\rho_{3} +
\rho_2\rho_1)|_{q=2}
\end{align*}
\begin{align*}
\lower 5pt \hbox{$\rho\circ ch$} \left(\lower 4pt
\hbox{$\chi^{\begin{tikzpicture} \filldraw [black] (0,0) circle
(1.5pt) (.5,0) circle (1.5pt)  (1,0) circle (1.5pt) (1.5,0) circle
(1.5pt); \node at (0, -.2) {$\scscs 1$}; \node at (0.5, -.2)
{$\scscs 2$}; \node at (0.75, .75) {$\scs 1$}; \node at (1, -.2)
{$\scscs 3$};  \node at (1.5, -.2) {$\scscs 4$}; \draw (0,0) ..
controls (.35,.6) and (1.15,.6).. (1.5,0);
\end{tikzpicture}} \uparrow_{U_4}^{G_4}$} \right) = (\rho_{4} +
2\rho_3\rho_1+\rho_2\rho_1^2)|_{q=2}
\end{align*}
\begin{align*}
\lower 5pt \hbox{$\rho\circ ch$} \left(\lower 4pt
\hbox{$\chi^{\begin{tikzpicture} \filldraw [black] (0,0) circle
(1.5pt) (.5,0) circle (1.5pt)  (1,0) circle (1.5pt) (1.5,0) circle
(1.5pt); \node at (0, -.2) {$\scscs 1$}; \node at (0.5, -.2)
{$\scscs 2$}; \node at (0.5, .55) {$\scs 1$}; \node at (1, -.2)
{$\scscs 3$};  \node at (1, .55) {$\scs 1$}; \node at (1.5, -.2)
{$\scscs 4$}; \draw (0,0) .. controls (.25,.4) and (.75,.4)..
(1,0);\draw (.5,0) .. controls (.75,.4) and (1.25,.4).. (1.5,0);
\end{tikzpicture}} \uparrow_{U_4}^{G_4}$} \right)
= \lower 5pt \hbox{$\rho\circ ch$} \left(\lower 4pt
\hbox{$\chi^{\begin{tikzpicture} \filldraw [black] (0,0) circle
(1.5pt) (.5,0) circle (1.5pt)  (1,0) circle (1.5pt) (1.5,0) circle
(1.5pt); \node at (0, -.2) {$\scscs 1$}; \node at (0.5, -.2)
{$\scscs 2$}; \node at (0.75, .75) {$\scs 1$};\node at (1, -.2)
{$\scscs 3$}; \node at (0.75, .35) {$\scs 1$};\node at (1.5, -.2)
{$\scscs 4$}; \draw (0,0) .. controls (.35,.7) and (1.15,.7)..
(1.5,0); \draw (.5,0) .. controls (.625,.3) and (.875,.3).. (1,0);
\end{tikzpicture}} \uparrow_{U_4}^{G_4}$} \right) =  (2\rho_{4} +
\rho_2\rho_2+\rho_3\rho_1)|_{q=2}
\end{align*}
\begin{align*}
\lower 5pt \hbox{$\rho\circ ch$} \left(\lower 4pt
\hbox{$\chi^{\begin{tikzpicture} \filldraw [black] (0,0) circle
(1.5pt) (.5,0) circle (1.5pt)  (1,0) circle (1.5pt) (1.5,0) circle
(1.5pt); \node at (0, -.2) {$\scscs 1$}; \node at (0.5, -.2)
{$\scscs 2$}; \node at (0.25, .4) {$\scs 1$}; \node at (1, -.2)
{$\scscs 3$};  \node at (1, .55) {$\scs 1$}; \node at (1.5, -.2)
{$\scscs 4$}; \draw (0,0) .. controls (.125,.2) and (.375,.2)..
(.5,0);\draw (.5,0) .. controls (.75,.4) and (1.25,.4).. (1.5,0);
\end{tikzpicture}} \uparrow_{U_4}^{G_4}$} \right)
= (\rho_{4} + \rho_3\rho_1)|_{q=2}
\end{align*}
\end{example}

Inspired from these results, we give the following conjecture and
open questions.

\begin{conjecture}
For a fixed finite field $\mathbb F_q$ and a supercharacter $\chi$
of $U_n$, we have
\begin{equation*}
\rho\circ ch(\chi\uparrow_{U_n}^{G_n}) \in \mathbb
N[\rho_1,\ldots,\rho_n].
\end{equation*}
\end{conjecture}
If the above conjecture is ture, then the following remark is
meaningful.
\begin{remark}
For a fixed finite field $\mathbb F_q$ and a character $\chi$ of
$U_n$, suppose $\nobreak{ch(\chi\uparrow_{U_n}^{G_n})} =
\sum_{\lambda\vdash n} C_{\lambda}\rho_{\lambda}$ where
$C_{\lambda}\in\mathbb C$. We have
\begin{equation}\label{eq:dim}
\dim(\chi) = \sum_{\lambda\vdash n} C_{\lambda}.
\end{equation}
\end{remark}
\begin{proof}
From Corollary~\ref{col:multilichar} we have
\begin{equation*}
\chi\uparrow_{U_n}^{G_n} = \sum_{\lambda\vdash n} C_{\lambda}
(\chi^{\lambda_1 |\lambda_2 | \ldots |
\lambda_l}_{\vec{q_1},\ldots,\vec{q_{l}}}\uparrow_{U_n}^{G_n}) =
\bigg(\sum_{\lambda\vdash n} C_{\lambda} \chi^{\lambda_1 |\lambda_2
| \ldots |
\lambda_l}_{\vec{q_1},\ldots,\vec{q_{l}}}\bigg)\uparrow_{U_n}^{G_n},
\end{equation*}
where $\vec{q_i}=(q_{i,1},\ldots,q_{i,\lambda_i-1})\in (\mathbb
F_q^{\times})^{\lambda_i-1}$. So we have
\begin{equation*}
\dim(\chi) = \dim\bigg(\sum_{\lambda\vdash n} C_{\lambda}
\chi^{\lambda_1 |\lambda_2 | \ldots |
\lambda_l}_{\vec{q_1},\ldots,\vec{q_{l}}}\bigg) =
\sum_{\lambda\vdash n} C_{\lambda} \dim(\chi^{\lambda_1 |\lambda_2 |
\ldots | \lambda_l}_{\vec{q_1},\ldots,\vec{q_{l}}}).
\end{equation*}
Since $\dim(\chi^{\lambda_1 |\lambda_2 | \ldots |
\lambda_l}_{\vec{q_1},\ldots,\vec{q_{l}}}) = 1$, we prove the
remark.
\end{proof}

\begin{question}
For a fixed finite field $\mathbb F_q$ and a supercharacter $\chi$
of $U_n$, try to find a formula for the plethysms of the
characteristic map of $\chi\uparrow_{U_n}^{G_n}$.
\begin{equation*}
\rho\circ ch(\chi\uparrow_{U_n}^{G_n})= \sum_{\lambda\vdash n}
C_{\lambda}\rho_{\lambda},
\end{equation*}
where
$\rho_{\lambda}=\rho_{\lambda_1}\rho_{\lambda_1}\ldots\rho_{\lambda_l}$
for $\lambda={\lambda_1,\ldots,\lambda_l}$. It is nice to give a
combinatorial formula for the coefficient $C_{\lambda}$ since the
example above suggest a few possible rules.
\end{question}
\begin{remark}
If we have the formula of $\rho\circ ch(\chi\uparrow_{U_n}^{G_n})$,
we can easily get the expression for the characteristic map of
$\chi\uparrow_{U_n}^{G_n}$ in terms of basis $\{\tilde P_{\bm{\mu}}
| \bm{\mu}\in\mathcal{P}^{\Phi} \}$ by Remark~\ref{rm:equi}. We may
also use $\hat{\rho}$ to get an expression in the basis
$\{S_{\bm{\lambda}} | \mathbf{\bm{\lambda}}\in\mathcal{P}^{\Theta}
\}$ by Corollary~\ref{col:irrdecomp}.
\end{remark}

\begin{question}
Up to now the induced representations that we are considering are in
characteristic zero. Another problem we can think about is what
happens in characteristic $p$ case.
\end{question}


\parindent=0pt

\begin{thebibliography}{XX}

\bibitem{An95}
{\sc C. Andr\'e}, {\em Basic characters of the unitriangular group},
J. Algebra \textbf{175} (1995), 287--319.

\bibitem{MulAu}
{\sc Marcelo Aguiar, Carlos Andre, Carolina Benedetti, Nantel
Bergeron, Zhi Chen, Persi Diaconis, Anders Hendrickson, Samuel
Hsiao, I. Martin Isaacs, Andrea Jedwab, Kenneth Johnson, Gizem
Karaali, Aaron Lauve, Tung Le, Stephen Lewis, Huilan Li, Kay
Magaard, Eric Marberg, Jean-Christophe Novelli, Amy Pang, Franco
Saliola, Lenny Tevlin, Jean-Yves Thibon, Nathaniel Thiem, Vidya
Venkateswaran, C. Ryan Vinroot, Ning Yan, Mike Zabrocki.}, {\em
Basic characters of the unitriangular group}, To appear in Adv.
Math., DOI:10.1016/j.aim.2011.12.024.

\bibitem{Gelgra}
{\sc I. M. Gelfand and M. I. Graev},{\em Construction of irreducible
representations of simple algebraic groups over a finite field},
Dokl. Akad. Nauk SSSR 147 (1962), 529每532.

\bibitem{Green}
{\sc J. A. Green}, {\em The Characters of the finite general linear
groups}, Transactions of the American Mathematical Society,
\textbf{80} (1955), 402--447.

\bibitem{Kawa}
{\sc N. Kawanaka}, {\em Generalized Gel＊fand-Graev representations
and Ennola duality}, In Algebraic groups and related topics
(Kyoto/Nagoya, 1983), 175每206, Adv. Stud. Pure Math., 6, North-
Holland, Amsterdam, 1985.

\bibitem{Mac}
{\sc I.G. Macdonald},  {\em Symmetric Functions and
Hall-Polynomials, Oxford Mathematical Monographs}, Oxford Univ.
Press, second edition (1995) 488p.

\bibitem{Nat}
{\sc Nathaniel Thiem}, {\em Unipotent Hecke algebras: the
structure,representation theory, and combinatorics},
  Ph.D Thesis, University of Wisconsin - Madison, 2004.


\bibitem{Nat2}
{\sc  Nathaniel Thiem}, {\em Branching rules in the ring of
superclass functions  of unipotent upper-triangular matrices}, J.
Algebraic Combin.  {\bf 31}  (2010),  no. 2, 267--298.

\bibitem{NatVin}
{\sc Nathaniel Thiem and C. Ryan Vinroot}, {\sc Gelfand-Graev
characters of the finite unitary groups}, Electron. J. Combin. 16
(2009), no. 1, Research Paper 146, 37 pp.

\bibitem{Stein}
{\sc R. Steinberg}, {\em Lectures on Chevalley groups, mimeographed
notes}, Yale University, 1968.

\bibitem{Yan}
{\sc N. Yan}, {\em Representation theory of the finite unipotent
linear groups}, Unpublished Ph.D. Thesis, Department of Mathematics,
University of Pennsylvania, 2001.

\bibitem{Zele}
{\sc A. V. Zelevinsky}, {\em Representations of finite classical
groups. A Hopf algebra approach}, Lec- ture Notes in Mathematics
869, Springer每Verlag, Berlin每New York, 1981.

\end{thebibliography}
\end{document}